\documentclass[10pt, reqno]{amsart}
\usepackage{amsfonts,amssymb,latexsym}
\usepackage{amsmath,amsthm}
\usepackage{eucal}
\usepackage[latin1]{inputenc}
\usepackage{color}

\newtheorem{theorem}{Theorem}[section]
\newtheorem{lemma}[theorem]{Lemma}
\newtheorem{proposition}[theorem]{Proposition}
\newtheorem{corollary}[theorem]{Corollary}

\theoremstyle{definition}
\newtheorem{definition}[theorem]{Definition}

\theoremstyle{remark}
\newtheorem{remark}{Remark}[section]

\numberwithin{equation}{section}

\begin{document}

\title[Local well-posedness for a full dispersion Boussinesq system]{On the local well-posedness for a full dispersion Boussinesq system with surface tension}
\author[H. Kalisch]{Henrik Kalisch}

\author[D. Pilod]{Didier Pilod}
\address{Department of Mathematics, University of Bergen, Postbox 7800, 5020 Bergen, Norway}
\email{Henrik.Kalisch@uib.no}
\email{Didier.Pilod@uib.no}


\date{\today}

\maketitle

\begin{abstract} 
In this note, we prove local-in-time well-posedness for a fully dispersive Boussinesq system arising 
in the context of free surface water waves in two and three spatial dimensions.
Those systems can be seen as a weak nonlocal dispersive perturbation of the shallow-water system. 
Our method of proof relies on energy estimates and a compactness argument. 
However, due to the lack of symmetry of the nonlinear part, those traditional methods have to be supplemented
with the use of a modified energy in order to close the {\it a priori} estimates.
\end{abstract}

\section{Introduction}

Consideration is given to the one-dimensional fully dispersive Boussinesq system 
\begin{equation} \label{WB}
\left\{\begin{array}{l}
\partial_t \eta+\mathcal{K}(D)\partial_xu+\partial_{x}(\eta u)=0 \, , \\ 
\partial_t u+\partial_x \eta+u\partial_xu=0 \, ,
\end{array}\right.
\end{equation}
where $x \in \mathbb R$, $t \in \mathbb R$, $\eta(x,t) \in \mathbb R$ and $u(x,t) \in \mathbb R$, and its two-dimensional counterpart 
\begin{equation} \label{WB2d}
\left\{\begin{array}{l}
\partial_t \eta+\mathcal{K}(D) \nabla \cdot \textbf{u}+\nabla \cdot (\eta \textbf{u})=0 \, , \\ 
\partial_t \textbf{u}+\nabla \eta+\frac12 \nabla |\textbf{u}|^2=0 \, ,
\end{array}\right.
\end{equation}
for $x \in \mathbb R^2$, $t \in \mathbb R$, $\eta(x,t) \in \mathbb R$ and $\textbf{u}(x,t) \in \mathbb R^2$, where $\mathcal{K}(D)$ is a nonlocal operator related to the dispersion of the linearized water-wave system 
in finite depth. Namely, $\mathcal{K}(D)$ is defined as a Fourier multiplier associated with the symbol 
\begin{equation} \label{K_symbol}
K(\xi)= \frac{\tanh(|\xi|)}{|\xi|} \big(1+\beta |\xi|^2 \big) \, ,
\end{equation}
where $\beta$ is a nonnegative dimensionless number related to the surface tension (see \cite{ReKa}).

\medskip
Those systems were proposed in \cite{La,AcMiPa,KLPS} as approximate models for the study of surface water waves, 
and provide a two-directional alternative to the well known Whitham equation. 
We also refer to \cite{HuPa,DIT,DDKM,Di} for other versions of full-dispersion Boussinesq type systems. 
The unknowns $\eta$ and $u$ in \eqref{WB} represent respectively 
the deflection of the free surface from its equilibrium position ($\eta=0$) and the velocity at the free surface, while the bottom is assumed to be at constant depth $h=-1$. 

The one-dimensional Whitham equation 
\begin{equation} \label{Whitham} 
\partial_tu+\mathcal{W}(D)\partial_xu+u\partial_xu=0 \, ,
\end{equation}
where $x \in \mathbb R$, $t \in \mathbb R$, $u=u(x,t) \in \mathbb R$ and $\mathcal{W}(D)$ is the Fourier multiplier associated with the symbol $W(\xi):=\sqrt{K(\xi)}$, was introduced by Whitham in \cite{Whit} as an alternative to the Korteweg-de Vries (KdV) equation by keeping the exact dispersion of the linearized water waves system in finite depth. This equation has drawn quite a bit of attention lately. In particular, it displays, in the case of pure gravity waves ($\beta=0$), several interesting phenomena already predicted by Whitham: a solitary wave regime close to KdV \cite{EhGrWa}, the existence of a wave of greatest height (Stokes wave) \cite{EhWa}, the existence of shocks \cite{Hu}, and modulational instability of 
steady periodic waves \cite{HuJo,Sanford}. Note that when surface tension is taken into account ($\beta>0$), the dynamics of \eqref{Whitham} appears to be completely different (see \cite{KLPS} and the references therein). Moreover, it was proved to be a relevant water wave model in the long wave regime on the same time scale as  the KdV equation \cite{La,KLPS}.\footnote{More precisely it was proved to be consistent with the KdV equation on those time scales.}
We also refer to \cite{Carter,TKCO,KLPS} for other interesting numerical simulations. 

\smallskip
Returning to the full-dispersion systems \eqref{WB} and \eqref{WB2d}, it has been shown in \cite{EhPeWa} that \eqref{WB}  is locally well-posed in the case of pure gravity waves ($\beta=0$) if one makes the assumption that the initial elevation $\eta$ is bounded by below by a positive constant. No results seem to be known when surface tension is taken into account, i.e. in the case $\beta>0$.

Our main result in this note is a proof of well posedness for systems \eqref{WB} and \eqref{WB2d} 
in the case of $\beta>0$ under a non-cavitation assumption on the initial surface elevation $\eta(\cdot,0)$.

\begin{definition} 
Let $d=1$ or $2$ and $s> \frac{d}2$. We say that the initial elevation $\eta_0 \in H^s(\mathbb R^2)$ satisfies the \emph{non-cavitation condition} if 
\begin{equation} \label{noncavitation}
\exists \, h_0 \in (0,1) \quad \text{such that} \quad 1+\eta_0(x) \ge h_0 , \ \forall \, x \in \mathbb R^ d \, .
\end{equation}
\end{definition}

\begin{remark} 
The non-cavitation condition \eqref{noncavitation} is a physical condition meaning that the elevation of the initial wave cannot touch the bottom of the fluid. 
\end{remark}
\begin{theorem} \label{LWP}
Assume that $\beta>0$. 

\noindent (i) Let $s >\frac52$. Let  $(\eta_0,u_0)  \in H^{s}(\mathbb R) \times H^{s+\frac12}(\mathbb R)$ satisfying the non-cavitation condition \eqref{noncavitation}. 
Then, there exists a positive time $T=T(\|(\eta_0,u_0)\|_{H^s \times H^{s+\frac12}})$, (which can be chosen as a non-decreasing function of its argument), and a unique solution $(\eta,u)$ to \eqref{WB} satisfying 
\begin{equation} \label{LWP.1}
(\eta,u) \in C\big([0,T] : H^{s}(\mathbb R) \times H^{s+\frac12}(\mathbb R)\big) \quad \text{and}  \quad (\eta(\cdot,0),u(\cdot,0))=(\eta_0,u_0)\, . 
\end{equation} 
In addition, the flow function mapping initial data to solutions is continuous.

\smallskip
\noindent (ii) Let $s >3$. Let  $(\eta_0,\textbf{u}_0)  \in H^{s}(\mathbb R^2) \times H^{s+\frac12}(\mathbb R^2)^2$ satisfying the non-cavitation condition \eqref{noncavitation} and such that $\text{curl}\, \textbf{u}_0=0$. Then, there exists a positive time $T=T(\|(\eta_0,\textbf{u}_0)\|_{H^s \times H^{s+\frac12}\times H^{s+\frac12}})$,  (which can be chosen as a non-decreasing function of its argument), and a unique solution $(\eta,u)$ to \eqref{WB2d} satisfying 
\begin{equation} \label{LWP.2}
(\eta,\textbf{u}) \in C\big([0,T] : H^{s}(\mathbb R) \times H^{s+\frac12}(\mathbb R)^2\big) \quad \text{and}  \quad (\eta(\cdot,0),\textbf{u}(\cdot,0))=(\eta_0,\textbf{u}_0)\, . 
\end{equation} 
In addition, the flow function mapping initial data to solutions is continuous.
\end{theorem}

\begin{remark} 
The same results also hold in the periodic case. The proof is similar up to small changes in the commutator estimates (see for example \cite{KePi}).
\end{remark}

\begin{remark} 
The time of existence $T$ in Theorem \ref{WB} with respect to the parameter $\beta$ 
satisfies $T(\beta) \lesssim \beta$. In particular $T(\beta) \to 0$, when $\beta \to 0$. 
Note that in the case of pure gravity waves ($\beta=0$), system \eqref{WB} is probably 
ill-posed\footnote{We refer to \cite{KLPS} for an heuristic argument of this fact.} 
unless one makes the nonphysical assumption that $\eta \ge c_0>0$ as in \cite{EhPeWa}. 
One interesting observation is that the present situation appears to be similar to
the case of the nonlinear Kevin-Helmholtz problem for two-fluid interfaces,
where the criterion established in \cite{La2} explains why capillarity is necessary 
for the well-posedness of the system, but does not affect the long-time dynamics.

For the sake of simplicity, we will renormalize the system and assume that $\beta=1$ in the following.
\end{remark}

\begin{remark} We do not consider here the system in the long-wave regime as it was done in \cite{KLPS}, 
since our method of proof does not seem to provide, at least directly, good lower bounds 
for the existence time with respect to the small parameter $\epsilon$ measuring the 
size of the dispersive and nonlinear effects, which are of the same order in this regime. 
It remains nevertheless an interesting issue to  prove that systems \eqref{WB} and \eqref{WB2d} 
are locally well posed over large time as it was done for some of the $(a,b,c,d)$-Boussinesq 
systems \cite{SaXu,Bu,SaWaXu}.   
\end{remark}

\medskip 

The proof of Theorem \ref{LWP} is based on energy estimates and a standard compactness argument. The main difficulty  lies in the lack of  symmetry of the nonlinearity in \eqref{WB}. Indeed, a direct energy estimate at the $H^s \times H^{s+\frac12}$ level\footnote{The scaling $H^s \times H^{s+\frac12}$ is needed to cancel out the linear terms} in the one dimensional case gives only (for $s$ large enough)
\begin{equation} \label{direct_ee}
\begin{split}
&\frac{d}{dt} \big( \|\eta(t)\|_{H^s}^2 +\|u(t)\|_{H^{s+\frac12}}^2 \big) \\&\lesssim \big(1+\|\eta(t)\|_{H^s}+\|u(t)\|_{H^{s+\frac12}}\big)\big( \|\eta(t)\|_{H^s}^2 +\|u(t)\|_{H^{s+\frac12}}^2\big)
 +\left| \int \eta J_x^s\partial_xu J^s_x\eta\right|  
\end{split}
\end{equation}
where $J^s_x$ denotes the Bessel potential of order $-s$. Note that the last term on the right-hand side of \eqref{direct_ee} cannot be handled directly by integration by parts or commutator estimates. 

In the absence of dispersion, it is well-known that one can symmetrize the system by using hyperbolic symmetrizers. We refer for example to \cite{La} for the shallow water system. This technique can be adapted in a nontrivial way when one adds a local dispersive perturbation to the system \cite{SaXu,SaWaXu}. However, it is not clear whether it still applies for the systems \eqref{WB} and \eqref{WB2d}\footnote{Note that the technique may work for some other systems with a nonlocal dispersion. We refer for example to \cite{Xu} for a nonlocal dispersive system in the context of internal wave. }. 

Here, we use instead a modified-energy method. The idea is to add the lower-order cubic term $\int \eta (J^s_xu)^2$ to the energy. The linear contribution of the derivative of this term will cancel out the last term on the right-hand side of \eqref{direct_ee}, while the contribution coming from the nonlinear terms can easily be controlled. This approach enables us to close the energy estimate. A similar argument can be used to derive an energy estimate for the difference of two solutions. Once these estimates are established, the proof proceeds using bootstrapping and classical compactness arguments. Finally, it is worth noticing that the non-cavitation condition on the initial data, which propagates through the flow of \eqref{WB}, is needed to ensure the coercivity of the modified energy.

The proof in the 2 dimensional case is very similar. This time the energy needs  to be modified by the term $\int \eta |J^s_x\textbf{u}|^2$.
Moreover, we also need to assume a curl-free condition on the initial velocity $\textbf{u}_0$. Note that this condition is preserved by the flow of \eqref{WB2d}. When $\textbf{u}$ is curl-free, the term $\frac12 \nabla |\textbf{u}|^2$ can be written as two transport terms, namely $(\textbf{u}\cdot \nabla u_1,\textbf{u} \cdot \nabla u_2)^{T}$, where $\textbf{u}=(u_1,u_2)^{T}$. This fact enables us to close the energy estimates for this term by using the Kato-Ponce commutator estimates (see for example Lemma 4.2. in \cite{LiPiSa}). Note that our proof would also work, without the curl-free assumption on $\textbf{u}$, when considering a nonlinearity of the form $\big(\textbf{u} \cdot \nabla\big) \textbf{u}$ instead of $\nabla |\textbf{u}|^2$ in the second line of \eqref{WB2d}. 
For the sake of simplicity, we will focus below on the proof in the one-dimensional case and will indicate in the last section what are the main changes in the two-dimensional case.

The use of a modified energy is well-known to be  a  powerful tool in the study of nonlinear partial differential equations. We refer among others to \cite{Kw,KePi} (well-posedness results for higher-order KdV type equations), \cite{HITW} (long time existence results for small initial data for the Burgers-Hilbert equation) and \cite{PlTzVi} (growth of Sobolev norm for NLS) for some applications of the modified energy methods in related contexts.  The method of proof introduced here seems to be quite general and we hope that it will have further applications to other weakly dispersive and nonlocal
perturbations of nonlinear hyperbolic systems.

\medskip 
The paper is organized as follows: in Section 2, we give the notations and recall some commutator estimates. Section 3 and 4 are devoted to the proof of the energy estimates respectively for a solution and for the difference of two solutions. Finally, we give the proof of Theorem \ref{LWP} (i) in Section 5 and explain the main changes for the two-dimensional case in Section 6.

\section{Notations and preliminary estimates}

\subsection{Notations} 

\begin{itemize}
\item{} Throughout the text, $c$ will denote a positive constant which may change from line to line. Also, for any positive numbers $a$ and $b$, the notation $a \lesssim b$ means that $a \le c b$. 

\item{} The operator $\mathcal{F}$ denotes the Fourier transform. We often write $\mathcal{F}(f)(\xi)=\widehat{f}(\xi)$.

\item{} In one dimension, $\mathcal{H}$ will denote the Hilbert transform, \text{i.e} $\big( \mathcal{H}f \big)^{\wedge}(\xi)=-i \text{sgn}(\xi) \widehat{f}(\xi)$.

\item{} In two dimensions, $\mathcal{R}_j$, $j=1,2$, will denote the Riesz transforms, \textit{i.e.} $\big( \mathcal{R}_jf \big)^{\wedge}(\xi)=-i \frac{\xi_j}{|\xi|} \widehat{f}(\xi)$.

\item{} For any $\alpha \in \mathbb R$, $D_x^{\alpha}$  will denote the Riesz potential of order $-\alpha$, defined via Fourier transform by 
$\big(D_x^{\alpha}f \big)^{\wedge}(\xi)=|\xi|^{\alpha} \widehat{f}(\xi)$. In particular, it follows that $D_x^1=\mathcal{H}\partial_x$.

\item{} For any $\alpha \in \mathbb R$, $J_x^{\alpha}$ will denote the Bessel potential of order $-\alpha$, defined via Fourier transform by 
$\big(J_x^{\alpha}f \big)^{\wedge}(\xi)=(1+\xi^2)^{\frac{\alpha}2} \widehat{f}(\xi)$. In particular, it is well-known that the $L^2$-based Sobolev space $H^s$ can be defined by the norm $\|f\|_{H^s}=\|J^s_xf\|_{L^2}$.
\item{} If $A$ and $B$ are two operators, then $[A,B]$denotes the commutator between $A$ and $B$, \text{i.e.} $[A,B]f=ABf-BAf$.

\end{itemize}

\subsection{Fourier multiplier} We reformulate system \eqref{WB} as 
\begin{equation} \label{WBbis}
\left\{\begin{array}{l}
\partial_t \eta+\mathcal{M}(D)(1-\partial_x^2)u-\mathcal{H}u+\mathcal{H}\partial_x^2u+\partial_{x}(\eta u)=0 \, , \\ 
\partial_t u+\partial_x \eta+u\partial_xu=0 \, ,
\end{array}\right.
\end{equation}
where $\mathcal{H}$ is the Hilbert transform and $\mathcal{M}(D)$ is the Fourier multiplier associated to the symbol 
\begin{equation} \label{def:M}
M(\xi)=i\big(\tanh(\xi)-\text{sgn} \, (\xi)\big) \, .
\end{equation}
By recalling the pointwise estimate (see for example \cite{HuPa})
\begin{displaymath}
\big|\tanh(\xi)-\text{sgn} \, (\xi)\big| \le e^{-|\xi|}, \quad \forall \, \xi \in \mathbb R \, ,
\end{displaymath}
it follows easily from Plancherel identity that 
\begin{equation} \label{est:M}
\|J_x^s \mathcal{M}(D) f \|_{L^2} \lesssim \|f\|_{L^2} \, , \quad \forall \,s \in \mathbb R \, .
\end{equation}
Note that the implicit constant in the former inequality depends of course on $s$.
Moreover, we also have from Young's theorem on convolution 
\begin{displaymath} 
\|\mathcal{M}(D)(1-\partial_x^2) f \|_{L^{\infty}}=\|\big(M(\xi)(1+\xi^2)\big)^{\vee} \ast f \|_{L^{\infty}} \le \|\big(M(\xi)(1+\xi^2)\big)^{\vee} \|_{L^1} \|f\|_{L^{\infty}}  \, ,
\end{displaymath}
so that 
\begin{equation} \label{est:M_Linfty}
\|\mathcal{M}(D)(1-\partial_x^2) f \|_{L^{\infty}}\lesssim \|f\|_{L^{\infty}} \, ,
\end{equation}
since $: \xi \mapsto M(\xi)(1+\xi^2)$ is a Schwartz function.

\smallskip
Finally, we will also need  an estimate comparing the Bessel and Riesz potentials. We claim that
\begin{equation} \label{est:BR}
\|(J^1_x-D^1_x)\partial_xf \|_{L^2} \lesssim \|f\|_{L^2}  \, .
\end{equation}
Indeed, it follows from Plancherel's identity that 
\begin{displaymath}
\|(J^1_x-D^1_x)\partial_xf \|_{L^2}^2= \int |\xi|^2 \big| (1+\xi^2)^{\frac12}-|\xi|\big|^2 |\widehat{f}(\xi)|^2 d\xi  =  \int |\xi|^4 \big|\big(1+\frac1{\xi^2}\big)^{\frac12}-1  \big|^2|\widehat{f}(\xi)|^2d\xi \, ,
\end{displaymath}
which implies \eqref{est:BR}, since the function $:\xi \mapsto |\xi|^4 \big|\big(1+\frac1{\xi^2}\big)^{\frac12}-1  \big|^2$ is bounded on $\mathbb R$.

\subsection{Commutator estimates}

First, we state the Kato-Ponce commutator estimate \cite{KaPo}.
\begin{lemma}[Kato-Ponce commutator estimates] \label{Kato-Ponce}
Let $s \ge 1$, $p,\ p_2, \ p_3 \in (1,\infty)$ and $p_1, \ p_4 \in (1,\infty]$ be such that $\frac1p=\frac1{p_1}+\frac1{p_2}=\frac1{p_3}+\frac1{p_4}$ . Then, 
\begin{equation}  \label{Kato-Ponce1}
\|[J_x^s,f]g\|_{L^p} \lesssim \|\partial_xf\|_{L^{p_1}}\|J_x^{s-1}g\|_{L^{p_2}}+\|J_x^sf\|_{L^{p_3}}\|g\|_{L^{p_4}} \, ,
\end{equation}
for any $f, \, g $ defined on $\mathbb R$.
\end{lemma}

We also state the fractional Leibniz rule proved in the appendix of \cite{KPV3}.
\begin{lemma} \label{LeibnizRule} 
Let $\sigma =\sigma_1+\sigma_2 \in (0,1)$ with $\sigma_i \in (0,\sigma)$ and $p, \  p_1, \ p_2 \in (1,\infty)$ satisfy
$\frac1p=\frac1{p_1}+\frac1{p_2}$. Then, 
\begin{equation}  \label{LeibnizRule1}
\|D_x^{\sigma}(fg)-fD_x^{\sigma}g-gD_x^{\sigma}f\|_{L^p} \lesssim \|D_x^{\sigma_1}f\|_{L^{p_1}}\|D_x^{\sigma_2}g\|_{L^{p_2}}.
\end{equation}
Moreover, the case $\sigma_2=0$, $p_2=\infty$ is also allowed.
\end{lemma}

The following commutator estimate was derived in Proposition 3.2 of \cite{DaMcPo}.
\begin{lemma} \label{comm_est}
Let $\alpha \in [0,1)$, $\beta \in (0,1)$ with $\alpha+\beta \in [0,1]$. Then, for any $p, q \in (1,\infty)$ and for any $\delta > 1/q$, there exists $c = c(\alpha; \beta; p; q; \delta) > 0$ such that
\begin{displaymath} 
\big\| D_x^{\alpha}[D^{\beta}_x,a]D_x^{1-(\alpha+\beta)}f\|_{L^p} \le c \|J^{\delta}_x\partial_xa\|_{L^q} \|f\|_{L^p} \, .
\end{displaymath}
\end{lemma}

\begin{corollary} \label{comm_est:coro}
Let $s>\frac32$. Then, 
\begin{equation} \label{comm_est.1}
\big\| [D^{\frac12}_x,a]D_x^{\frac12}f\big\|_{L^2} \lesssim \|a\|_{H^s} \|f\|_{L^2} \, .
\end{equation}
\end{corollary}

\begin{proof}
The proof of estimate \eqref{comm_est.1} follows directly  combining Lemma \ref{comm_est} with $p=2$, $\alpha=0$, $\beta=\frac12$ with the Sobolev embedding by choosing $q$  and $\delta$ such that $0<\delta-\frac1q<s-\frac32$. 
\end{proof}

\section{Energy estimates} The main goal of this section is to prove the following energy estimate for the solutions of \eqref{WB}.
\begin{proposition} \label{EE}
Let $s>2$ and $(\eta,u) \in C\big([0,T] : H^s(\mathbb R)\times H^{s+\frac12}(\mathbb R)\big)$ be a solution to \eqref{WB} on a time interval $[0,T]$ for some $T>0$.
Let us define the \emph{modified energy} $E^{s}(\eta,u)$ by 
\begin{equation} \label{EE.1}
E^{s}(\eta,u)(t)=\frac12 \|\eta(t)\|_{H^s}^2+\frac12 \|u(t)\|_{H^{s+\frac12}}^2+\frac12\int \eta (J^s_xu)^2(t) \, ,
\end{equation}
for all $t \in [0,T]$. Assume moreover that $\eta$ satisfies the condition 
\begin{equation} \label{noncavitation.1}
\exists \, \widetilde{h}_0 \in (0,1), \ h_1>0 \quad \text{such that} \quad  \widetilde{h}_0-1 \le \eta(x,t)  \le h_1, \ \forall \, (x,t) \in \mathbb R^2 \, .
\end{equation}
Then, the following estimates hold true for all $t \in [0,T]$.

\medskip
(1) \underline{Coercivity}.   
\begin{equation} \label{EE.2}
\frac12\big(\|\eta\|_{H^s}^2+c_0\|u\|_{H^{s+\frac12}}^2\big) \le E^s(\eta,u)\le \frac12 \big(\|\eta\|_{H^s}^2+(1+h_1)\|u\|_{H^{s+\frac12}}^2\big)  \, ,
\end{equation}
where $c_0=c_0(\widetilde{h}_0)$ is a positive constant.

\medskip
(2) \underline{Energy estimate}.
\begin{equation} \label{EE.3}
\frac{d}{dt}E^s(\eta,u) \lesssim E^s(\eta,u)+E^s(\eta,u)^2 \, .
\end{equation}
\end{proposition}

\begin{proof} We observe by using  condition \eqref{noncavitation.1} and Plancherel's identity that 
\begin{equation*}
\begin{split}
\int (J^{s+\frac12}_xu)^2 +\int \eta (J^s_xu)^2 &\ge \widetilde{h}_0 \int (J_x^su)^2+\int (1+\xi^2)^s \big( (1+\xi^2)^{\frac12}-1\big) |\widehat{u}(\xi)|^2 \\ 
& \ge \widetilde{h}_0 \int (J_x^su)^2+\widetilde{c}_0\int_{|\xi| \ge 1}(1+\xi^2)^{s+\frac12} |\widehat{u}(\xi)|^2 \, ,
\end{split}
\end{equation*}
where $\widetilde{c}_0$ is a universal constant depending only $s$. This implies the first inequality in \eqref{EE.2} in view of the definition of \eqref{EE.1} by choosing $c_0:=\min\{\widetilde{h}_0,\widetilde{c}_0\}$. The proof of the second inequality in \eqref{EE.2} is a direct consequence of \eqref{EE.1} and \eqref{noncavitation.1}.

To prove estimate \eqref{EE.3}, we will work on the reformulated version \eqref{WBbis} of \eqref{WB}. We compute the time derivative of each term on the left-hand side of \eqref{EE.3} separately. 
\smallskip

First, we get by using the Cauchy-Schwarz inequality, \eqref{est:M} and the identity $D^1_x=\mathcal{H}\partial_x$ that
\begin{displaymath} 
\begin{split}
\frac12&\frac{d}{dt}\int (J^s_x \eta)^2\\ &=-\int J_x^s\mathcal{M}(D)(1-\partial_x^2)u J^s_x \eta+\int J^s_x\mathcal{H}u J^s_x\eta-\int J^s_x\mathcal{H}\partial_x^2uJ^s_x\eta-\int J_x^s\partial_x(\eta u)J^s_x\eta \\ 
& \le c\|u\|_{H^s}\|\eta\|_{H^s}-\int J^s_xD^1_x\partial_xuJ^s_x\eta-\int J_x^s\partial_x(\eta u)J^s_x\eta \, .
\end{split}
\end{displaymath}
Moreover, it follows after integration by parts that 
\begin{displaymath} 
\begin{split}
\frac12\frac{d}{dt}\int (J^{s+\frac12}_xu)^2&=-\int J^{s+\frac12}_x\partial_x\eta J^{s+\frac12}_xu-\int J_x^{s+\frac12}(u \partial_xu)J^{s+\frac12}_xu \\ 
&=\int J^{s}_x\eta J^{s}_x(J_x^1-D_x^1)\partial_xu+\int J^{s}_x\eta J^{s}_xD_x^1\partial_xu- \int J_x^{s+\frac12}(u \partial_xu)J^{s+\frac12}_xu \, .
\end{split}
\end{displaymath}
Hence, we deduce by using \eqref{est:BR} that 
\begin{equation} \label{EE.4}
\begin{split}
\frac12\frac{d}{dt}&\big(\|\eta(t)\|_{H^s}^2+\|u(t)\|_{H^{s+\frac12}}^2\big)\\ & \le c\|u\|_{H^s}\|\eta\|_{H^s}-\int J_x^s\partial_x(\eta u)J^s_x\eta 
- \int J_x^{s+\frac12}(u \partial_xu)J^{s+\frac12}_xu \, .
\end{split}
\end{equation}

Now, we deal with the nonlinear terms appearing on the right-hand side of \eqref{EE.4}. First, we observe that 
\begin{displaymath} 
\begin{split}
 \int J_x^s\partial_x(\eta u)J^s_x\eta 
 &=  \int J_x^s(\eta \partial_x u)J^s_x\eta+\int J_x^s(\partial_x\eta u)J^s_x\eta \\ 
 & = \int [J_x^s, \eta] \partial_x u J^s_x\eta+\int \eta J^s_x \partial_xu J^s_x\eta+\int [J_x^s,u]\partial_x\eta J^s_x\eta+\int u J^s_x\partial_x\eta J^s_x\eta \, .
 \end{split}
\end{displaymath}
On the one hand, we get by using the commutator estimate \eqref{Kato-Ponce1}
\begin{displaymath}
\left|\int [J_x^s, \eta] \partial_x u J^s_x\eta \right|+\left| \int [J_x^s,u]\partial_x\eta J^s_x\eta \right| \lesssim \big( \|\partial_x\eta\|_{L^{\infty}}\|u\|_{H^s} +\| \partial_xu\|_{L^{\infty}} \| \eta\|_{H^s}\big)\| \eta\|_{H^s} \, .
\end{displaymath}
On the other hand, integration by parts and H\" older's inequality yield
\begin{displaymath}
\left|\int u J^s_x\partial_x\eta J^s_x\eta \right| \lesssim \| \partial_xu\|_{L^{\infty}} \|\eta\|_{H^s}^2 \, .
\end{displaymath}
Then, we deduce gathering the above estimates that 
\begin{equation} \label{EE.5}
\begin{split}
 \int & J_x^s\partial_x(\eta u)J^s_x\eta\\ &=\int \eta J^s_x \partial_xu J^s_x\eta+\mathcal{O}\big(\big( \|\partial_x\eta\|_{L^{\infty}}\|u\|_{H^s} +\| \partial_xu\|_{L^{\infty}} \| \eta\|_{H^s}\big)\| \eta\|_{H^s} \big) \, .
 \end{split}
\end{equation}

To deal with the second one,  we get integrating by parts that 
\begin{displaymath} 
\begin{split}
 \int J_x^{s+\frac12}(u \partial_xu)J^{s+\frac12}_xu &= \int [J_x^{s+\frac12},u]\partial_xu J^{s+\frac12}_xu+ \int uJ_x^{s+\frac12} \partial_xuJ^{s+\frac12}_xu\\ &=\int [J_x^{s+\frac12},u]\partial_xu J^{s+\frac12}_xu-\frac12\int \partial_x u(J_x^{s+\frac12} u)^2 \, .
 \end{split}
\end{displaymath}
Then, it follows from the commutator estimate \eqref{Kato-Ponce1} and H\"older's inequality that 
\begin{equation} \label{EE.6} 
\left| \int J_x^{s+\frac12}(u \partial_xu)J^{s+\frac12}_xu \right| \lesssim \| \partial_xu \|_{L^{\infty}} \|u\|_{H^{s+\frac12}}^2 \, .
\end{equation}

Therefore, we conclude gathering \eqref{EE.4}, \eqref{EE.5} and \eqref{EE.6} and using the Sobolev embedding that 
\begin{equation} \label{EE.7}
\begin{split}
\frac12\frac{d}{dt}&\big(\|\eta\|_{H^s}^2+\|u\|_{H^{s+\frac12}}^2\big)\\ & \le -\int \eta J^s_x \partial_xu J^s_x\eta+c(1+\|\eta\|_{H^s})\|u\|_{H^s}\|\eta\|_{H^s}+c\|u\|_{H^s} \|u\|_{H^{s+\frac12}}^2 \, .
\end{split}
\end{equation}

Finally, we derive the cubic contribution of the energy with respect to time. By using \eqref{WBbis}, we get
\begin{equation} \label{EE.8}
\frac12\frac{d}{dt}\int \eta(J^s_x u)^2=\frac12\int \partial_t\eta(J^s_x u)^2+\int \eta J^s_x \partial_tu J_x^su=I_1+I_2+I_3 \, ,
\end{equation}
where 
\begin{displaymath}
I_1:=-\frac12\int \mathcal{M}(D)(1-\partial_x^2)u (J^s_x u)^2+\int \mathcal{H}u (J^s_x u)^2-\int \mathcal{H}\partial_x^2u(J^s_x u)^2-\int \partial_x(\eta u)(J^s_x u)^2 \, ,
\end{displaymath}
\begin{displaymath}
I_2:=-\int \eta J^s_x \partial_x \eta J^s_xu=\int \eta J^s_x \eta J^s_x\partial_xu+\int \partial_x\eta J^s_x \eta J^s_xu \, ,
\end{displaymath}
after integrating by parts, and 
\begin{displaymath}
I_3:=-\int \eta J^s_x(u\partial_xu)J_x^su \, .
\end{displaymath}

We have by using H\"older's inequality, \eqref{est:M_Linfty} and the Sobolev embedding that 
\begin{equation} \label{EE.9}
\begin{split}
|I_1| & \lesssim \big(\|u\|_{L^{\infty}}+\|\mathcal{H}u\|_{L^{\infty}}+\|\mathcal{H}\partial_x^2u\|_{L^{\infty}}+\|\partial_x(\eta u)\|_{L^{\infty}} \big)\|u\|_{H^s}^2
\\ & \lesssim \big( 1+\|\eta\|_{H^s} \big) \|u\|_{H^s}^3+c\|u\|_{H^{s+\frac12}}\|u\|_{H^s}^2 \, ,
\end{split}
\end{equation}
where we used the restriction $s +\frac12 > 2+\frac12$, \textit{i.e.} $s>2$.
Moreover, we observe that $I_2$ will cancel out with the first term on the right-hand side of \eqref{EE.7}. This is why we modify the energy by the cubic term $\frac12\int \eta(J^s_x u)^2$.
We rewrite $I_3$ by using the commutator notation and integration by parts as
\begin{displaymath} 
I_3=-\int \eta [J^s_x,u]\partial_xu J^s_xu-\int \eta u J^s_x\partial_xu J^s_xu=-\int \eta [J^s_x,u]\partial_xu J^s_xu+\frac12 \int \partial_x(\eta u)(J_x^su)^2 \, .
\end{displaymath}
Then, it follows from the Kato-Ponce commutator estimate \eqref{Kato-Ponce1} and the Sobolev embedding that 
\begin{equation} \label{EE.10} 
|I_3| \le \| \eta \|_{L^{\infty}} \| [J^s_x,u]\partial_xu \|_{L^2} \|J^s_xu \|_{L^2}+ \|\partial_x(\eta u)\|_{L^{\infty}} \|J_x^su\|_{L^2}^2 \lesssim \|\eta\|_{H^s}\|u\|_{H^s}^3
  \, .
\end{equation}
Hence, we deduce gathering \eqref{EE.8}-\eqref{EE.10} that 
\begin{equation} \label{EE.11}
\frac12\frac{d}{dt}\int \eta(J^s_x u)^2 \le \int \eta J^s_x \eta J^s_x\partial_xu+c\big( 1+\|\eta\|_{H^s} \big) \|u\|_{H^s}^3+c\|u\|_{H^{s+\frac12}}\|u\|_{H^s}^2 \, .
\end{equation}

There, we conclude the proof of estimate \eqref{EE.3} combining \eqref{EE.7} and \eqref{EE.11} with \eqref{EE.2}. 
\end{proof}

\section{Estimates for the differences of two solutions}

In this subsection, we derive  energy estimates for the difference of two solutions $(\eta_1,u_1)$ and $(\eta_2,u_2)$ of \eqref{WBbis} in $H^1(\mathbb R) \times H^{\frac32}(\mathbb R)$. 

Let us define $(\widetilde{\eta},\widetilde{u})=(\eta_1-\eta_2,u_1-u_2)$. Then $(\widetilde{\eta},\widetilde{u})$ is a solution to 
\begin{equation} \label{WBdiff} 
\left\{\begin{array}{l}
\partial_t \widetilde{\eta}+\mathcal{M}(D)(1-\partial_x^2) \widetilde{u}-\mathcal{H} \widetilde{u}+\mathcal{H}\partial_x^2 \widetilde{u}+\partial_{x}(\eta_1  \widetilde{u}+ \widetilde{\eta}u_2)=0 \, , \\ 
\partial_t  \widetilde{u}+\partial_x  \widetilde{\eta}+\frac12\partial_x((u_1+u_2) \widetilde{u})=0 \, ,
\end{array}\right.
\end{equation}
where the symbol $M(\xi)$ of the Fourier multiplier $\mathcal{M}(D)$ is defined in \eqref{def:M}.
\medskip

\begin{proposition} \label{Est_diff}
Let  $s>\frac52$ and $(\eta_1,u_1), \, (\eta_2,u_2) \in C([0,T] : H^s(\mathbb R)\times H^{s+\frac12}(\mathbb R))$ be two solutions to \eqref{WBbis} on a time intervall $[0,T]$ for some $T>0$.  

Let $(\widetilde{\eta},\widetilde{u})=(\eta_1-\eta_2,u_1-u_2)$ denote the difference between the two solutions.
We define the \emph{modified energy} $\widetilde{E}(\widetilde{\eta},\widetilde{v})$ by
\begin{equation} \label{Est_diff.1}
2\widetilde{E}(\widetilde{\eta},\widetilde{v})(t)=\|\widetilde{\eta}(t)\|_{L^2}^2+\|\partial_x\widetilde{\eta}(t)\|_{L^2}^2+\|\widetilde{u}(t)\|_{L^2}^2+\|D^{\frac12}_x\partial_x\widetilde{u}(t)\|_{L^2}^2+ \int \eta_1(\partial_x\widetilde{u})^2 \, ,
\end{equation}
for all $t \in [0,T]$. Assume moreover that $\eta_1$ satisfies the condition \eqref{noncavitation.1}. Then, the following estimates hold true on $[0,T]$.

\medskip
(1) \underline{Coercivity}. There exists $\alpha_0>0$ such that
\begin{equation} \label{Est_diff.2}
\frac12\big(\|\widetilde{\eta}\|_{H^1}^2+c_0\|\widetilde{u}\|_{H^{\frac32}}^2\big) \le \widetilde{E}(\widetilde{\eta},\widetilde{v}) \le \frac12\big(\|\widetilde{\eta}\|_{H^1}^2+(1+h_1)\|\widetilde{u}\|_{H^{\frac32}}^2\big) \, ,
\end{equation}
where $c_0=c_0(\widetilde{h}_0)$ is a positive constant..

\medskip
(2) \underline{Energy estimate}.
\begin{equation} \label{Est_diff.3}
\frac{d}{dt}\widetilde{E}(\widetilde{\eta},\widetilde{v})  \lesssim  \Big(1+\|\eta_1\|_{H^s}+\|\eta_2\|_{H^s}+\|u_1\|_{H^s}+\|u_2\|_{H^s} \Big)^2 \big( \|\widetilde{\eta}\|_{H^1}^2+ \|\widetilde{u}\|_{H^{\frac32}}^2\big)\ .
\end{equation}
\end{proposition}

\begin{proof} The proof of estimate \eqref{Est_diff.2} is similar as the one of \eqref{EE.2}. 

To prove \eqref{Est_diff.3}, we compute separately the time derivative of each term on the right-hand side of \eqref{Est_diff.1}.
First, it follows directly by using \eqref{WBdiff} and integrating by parts that 
\begin{equation} \label{Est_diff.4}
\begin{split}
\frac12\frac{d}{dt} \int \widetilde{u}^2&=-\int \widetilde{u}\partial_x\widetilde{\eta}-\frac12 \int \widetilde{u} \partial_x((u_1+u_2)\widetilde{u}) 
\\ & \lesssim \|\widetilde{u}\|_{L^2}\|\partial_x \widetilde{\eta}\|_{L^2}+\big( \|\partial_xu_1\|_{L^{\infty}}+\|\partial_xu_2\|_{L^{\infty}} \big) \|\widetilde{u}\|_{L^2}^2 \, .
\end{split}
\end{equation}
By using \eqref{est:M}, integration by parts and H\"older's inequality, we get that 
\begin{equation} \label{Est_diff.5}
\begin{split}
\frac12\frac{d}{dt} \int \widetilde{\eta}^2&=-\int \widetilde{\eta}\mathcal{M}(D)(1-\partial_x^2)\widetilde{u}+ \int \widetilde{\eta} \mathcal{H}\widetilde{u}-\int \widetilde{\eta} \mathcal{H}\partial_x^2\widetilde{u}-\int \widetilde{\eta} \partial_{x}(\eta_1  \widetilde{u}+ \widetilde{\eta}u_2)
\\ & \lesssim \|\widetilde{u}\|_{H^1}\|\widetilde{\eta}\|_{H^1}+\big(\|\eta_1\|_{L^{\infty}}+ \|\partial_x\eta_1\|_{L^{\infty}} \big) \|\widetilde{u}\|_{H^1}\|\widetilde{\eta}\|_{L^2}+\|\partial_xu_2\|_{L^{\infty}}\|\widetilde{\eta}\|_{L^2}^2 \, .
\end{split}
\end{equation}
Now, we turn to the higher-order part of the $H^1 \times H^{\frac32}$ norm of $(\widetilde{\eta},\widetilde{u})$.  On the one hand, we have that 
\begin{displaymath} 
\begin{split}
\frac12\frac{d}{dt} \int (\partial_x\widetilde{\eta})^2&=-\int \partial_x\widetilde{\eta}\partial_x\mathcal{M}(D)(1-\partial_x^2)\widetilde{u}+ \int \partial_x\widetilde{\eta} \partial_x\mathcal{H}\widetilde{u}-\int \partial_x\widetilde{\eta} \mathcal{H}\partial_x^3\widetilde{u}\\ & \quad -\int \partial_x\widetilde{\eta} \partial_{x}^2(\eta_1  \widetilde{u}+ \widetilde{\eta}u_2) \, .
\end{split}
\end{displaymath}
To deal with the nonlinear term, we integrate by parts and use H\"older's inequality. It follows that 
\begin{displaymath}
\begin{split}
\int \partial_x\widetilde{\eta} \partial_{x}^2(\eta_1  \widetilde{u}+ \widetilde{\eta}u_2) &\le\int \eta_1 \partial_x\widetilde{\eta} \partial_x^2 \widetilde{u}
+c\big(\|\partial_x\eta_1\|_{L^{\infty}}+ \|\partial_x^2\eta_1\|_{L^{\infty}} \big) \|\widetilde{u}\|_{H^1}\|\widetilde{\eta}\|_{H^1}
\\ &\quad+ c\big(\|\partial_xu_2\|_{L^{\infty}}+ \|\partial_x^2u_2\|_{L^{\infty}} \big)\|\widetilde{\eta}\|_{H^1}^2 \, ,
\end{split}
\end{displaymath}
which implies together with \eqref{est:M} that 
\begin{equation} \label{Est_diff.6}
\begin{split}
\frac12\frac{d}{dt} \int (\partial_x\widetilde{\eta})^2 & \le c\|\widetilde{u}\|_{H^1}\|\widetilde{\eta}\|_{H^1}-\int \partial_x\widetilde{\eta} \mathcal{H}\partial_x^3\widetilde{u} -\int \eta_1 \partial_x\widetilde{\eta} \partial_x^2 \widetilde{u} \\ 
& \quad+c\big(\|\partial_x\eta_1\|_{L^{\infty}}+ \|\partial_x^2\eta_1\|_{L^{\infty}} \big) \|\widetilde{u}\|_{H^1}\|\widetilde{\eta}\|_{H^1}
\\ & \quad + c\big(\|\partial_xu_2\|_{L^{\infty}}+ \|\partial_x^2u_2\|_{L^{\infty}} \big)\|\widetilde{\eta}\|_{H^1}^2 \, .
\end{split}
\end{equation}
On the other hand, we compute 
\begin{displaymath} 
\frac12\frac{d}{dt} \int (D_x^{\frac12}\partial_x\widetilde{u})^2=-\int D_x^{\frac12}\partial_x\widetilde{u}D_x^{\frac12}\partial_x^2\widetilde{\eta}
-\int D_x^{\frac12}\partial_x\widetilde{u}D_x^{\frac12}\partial_x^2((u_1+u_2)\widetilde{u}) \, .
\end{displaymath}
By using the identity $D_x^1=\mathcal{H}\partial_x$ and integration by parts, we have
\begin{displaymath}
-\int D_x^{\frac12}\partial_x\widetilde{u}D_x^{\frac12}\partial_x^2\widetilde{\eta}=\int \mathcal{H}\partial_x^3\widetilde{u} \partial_x \widetilde{\eta} \, ,
\end{displaymath}
so that this term will cancel out with the second one on the right-hand side of \eqref{Est_diff.6}. Now, we deal with the nonlinear term. It follows by using the standard Leibniz rule that 
\begin{displaymath}
\begin{split}
\int D_x^{\frac12}\partial_x\widetilde{u}&D_x^{\frac12}\partial_x^2((u_1+u_2)\widetilde{u})\\ &=\int D_x^{\frac12}\partial_x\widetilde{u}D_x^{\frac12}\big(\partial_x^2(u_1+u_2)\widetilde{u}+2\partial_x^2(u_1+u_2)\partial_x\widetilde{u}+(u_1+u_2)\partial_x^2\widetilde{u} \big) \\ 
& =: \mathcal{I}_1+\mathcal{I}_2+\mathcal{I}_3 \, .
\end{split}
\end{displaymath}
We deduce from the fractional Leibniz rule \eqref{LeibnizRule1} that 
\begin{displaymath}
\begin{split}
|\mathcal{I}_1| &\lesssim \|D_x^{\frac12}\partial_x \widetilde{u} \|_{L^2} \big(\| \partial_x^2u_1\|_{L^{\infty}}+\| \partial_x^2u_2\|_{L^{\infty}}\big)\|D_x^{\frac12} \widetilde{u} \|_{L^2}\\ &\quad  + \|D_x^{\frac12}\partial_x \widetilde{u} \|_{L^2}(\| \partial_x^2D_x^{\frac12}u_1\|_{L^2}+\| \partial_x^2D_x^{\frac12}u_2\|_{L^2})\|\widetilde{u} \|_{L^{\infty}} 
\end{split}
\end{displaymath}
and 
\begin{displaymath}
\begin{split}
|\mathcal{I}_2| &\lesssim \|D_x^{\frac12}\partial_x \widetilde{u} \|_{L^2} \big(\| \partial_xu_1\|_{L^{\infty}}+\| \partial_xu_2\|_{L^{\infty}}\big)\|D_x^{\frac12}\partial_x \widetilde{u} \|_{L^2}\\ &\quad  + \|D_x^{\frac12}\partial_x \widetilde{u} \|_{L^2}(\| \partial_xD_x^{\frac12}u_1\|_{L^4}+\| \partial_xD_x^{\frac12}u_2\|_{L^4})\|\partial_x\widetilde{u} \|_{L^4} \, .
\end{split}
\end{displaymath}
Moreover, by using $\partial_x=-\mathcal{H}D^1_x$, the commutator notation and integration by parts, we get 
\begin{displaymath}
\begin{split}
\mathcal{I}_3&=\int D_x^{\frac12}\partial_x\widetilde{u} [D_x^{\frac12},u_1+u_2]\partial_x^2 \widetilde{u}
+\int D_x^{\frac12}\partial_x\widetilde{u} (u_1+u_2) D_x^{\frac12}\partial_x^2\widetilde{u} \\ 
& =-\int D_x^{\frac12}\partial_x\widetilde{u} [D_x^{\frac12},u_1+u_2]D_x^{\frac12} \mathcal{H}D_x^{\frac12}\partial_x \widetilde{u}
-\frac12 \int \partial_x(u_1+u_2) (D_x^{\frac12}\partial_x\widetilde{u})^2  \, .
\end{split}
\end{displaymath}
Hence, the commutator estimate \eqref{comm_est.1} and the Sobolev embedding yield
\begin{displaymath}
|\mathcal{I}_3|\lesssim (\| u_1\|_{H^s}+\| u_2\|_{H^s})\|D_x^{\frac12}\partial_x \widetilde{u} \|_{L^2}^2 \, .
\end{displaymath}
Therefore, we deduce gathering those estimates that 
\begin{equation} \label{Est_diff.7}
\frac12\frac{d}{dt} \int (D_x^{\frac12}\partial_x\widetilde{u})^2 \le \int \mathcal{H}\partial_x^3\widetilde{u} \partial_x \widetilde{\eta}+c(\| u_1\|_{H^s}+\| u_2\|_{H^s})\| \widetilde{u} \|_{H^{\frac32}}^2 \, .
\end{equation}
Finally, to deal with the third term on the right-hand side of \eqref{Est_diff.7}, we need to use the cubic part in the modified energy. Observe by using \eqref{WBbis} and \eqref{WBdiff} that 
\begin{displaymath} 
\frac12\frac{d}{dt}\int \eta_1(\partial_x \widetilde{u})^2=\frac12\int \partial_t\eta_1(\partial_x \widetilde{u})^2+\int \eta_1 \partial_x \partial_t\widetilde{u} \partial_x\widetilde{u}=\mathcal{J}_1+\mathcal{J}_2+\mathcal{J}_3 \, ,
\end{displaymath}
where 
\begin{displaymath}
\mathcal{J}_1=-\frac12\int \mathcal{M}(D)(1-\partial_x^2)u_1 (\partial_x \widetilde{u})^2+\int \mathcal{H}u_1  (\partial_x \widetilde{u})^2-\int \mathcal{H}\partial_x^2u_1 (\partial_x \widetilde{u})^2-\int \partial_x(\eta_1 u_1) (\partial_x \widetilde{u})^2 \, ,
\end{displaymath}
\begin{displaymath}
\mathcal{J}_2=-\int \eta_1 \partial_x^2 \widetilde{\eta} \partial_x\widetilde{u}=\int \eta_1 \partial_x \widetilde{\eta} \partial_x^2\widetilde{u} 
+\int \partial_x\eta_1 \partial_x \widetilde{\eta} \partial_x\widetilde{u}\, ,
\end{displaymath}
after integrating by parts, and 
\begin{displaymath}
\mathcal{J}_3=-\int \eta_1 \partial_x^2( (u_1+u_2)\widetilde{u})\partial_x \widetilde{u} \, .
\end{displaymath}

We have by using H\"older's inequality, \eqref{est:M_Linfty} and the Sobolev embedding that 
\begin{displaymath} 
\begin{split}
|\mathcal{J}_1|& \lesssim \big(\|u_1\|_{L^{\infty}}+\|\mathcal{H}u_1\|_{L^{\infty}}+\|\mathcal{H}\partial_x^2u_1\|_{L^{\infty}}+\|\partial_x(\eta_1 u_1)\|_{L^{\infty}} \big)\|\partial_x\widetilde{u}\|_{L^2}^2
\\ & \lesssim \big( \|u_1\|_{H^s}+\|\eta_1\|_{H^s}\|u_1\|_{H^s} \big) \|\partial_x\widetilde{u}\|_{L^2}^2 \, ,
\end{split}
\end{displaymath}
where we used the restriction $s> \frac52$.
Moreover, we observe that the first term on the right-hand side of $\mathcal{J}_2$ will cancel out with the third term on the right-hand side of \eqref{Est_diff.6}. To handle $\mathcal{J}_3$, we use the standard Leibniz rule and integration by parts to get
\begin{displaymath} 
\mathcal{J}_3=-\int \eta_1 \partial_x^2(u_1+u_2) \widetilde{u}\partial_x \widetilde{u}+ \int \big(-\frac32\eta_1\partial_x(u_1+u_2)+\frac12 \partial_x\eta_1(u_1+u_2)\big) (\partial_x \widetilde{u})^2 \, .
\end{displaymath}
It follows from H\"older and Sobolev inequalities that
\begin{displaymath} 
|\mathcal{J}_3| \lesssim \| \eta_1\|_{H^s} \big( \|u_1\|_{H^s}+\|u_2\|_{H^s} \big)\| \widetilde{u} \|_{H^1}^2 \, .
\end{displaymath}
Hence, we deduce gathering those estimates that 
\begin{equation} \label{Est_diff.8}
\frac12\frac{d}{dt}\int \eta_1(\partial_x \widetilde{u})^2 \le \int \eta_1 \partial_x \widetilde{\eta} \partial_x^2\widetilde{u}+c\big(1+\| \eta_1\|_{H^s}\big) \big( \|u_1\|_{H^s}+\|u_2\|_{H^s} \big)\| \widetilde{u} \|_{H^1}^2 \, .
\end{equation}

Therefore, we conclude the proof of \eqref{Est_diff.3} gathering \eqref{Est_diff.4}--\eqref{Est_diff.8}.

\end{proof}

\begin{remark}
Observe that the restriction $s>\frac52$ in Theorem \ref{LWP} (i) appears in Proposition \ref{EE}.
\end{remark}

\section{proof of Theorem \ref{LWP}}
We begin this section by proving an \textit{a priori} estimate on the solutions $(\eta,u)$ to \eqref{WB}. 
\begin{lemma} \label{apriori}
Let $s>\frac52$.  Let $(\eta,u) \in C\big([0,T^{\star}) : H^s(\mathbb R) \times H^{s+\frac12}(\mathbb R)\big)$ be a solution to \eqref{WB} corresponding to initial data $(\eta_0,u_0) \in H^s(\mathbb R) \times H^{s+\frac12}(\mathbb R)$ satisfying the non-cavitation condition \eqref{noncavitation}, defined on its maximal time of existence and satisfying the blow-up alternative: 
\begin{equation} \label{blow-up_alt}
\text{If} \quad T^{\star}<\infty \quad \text{then} \quad \displaystyle{\lim_{t \nearrow T^{\star}} \|(\eta(t),u(t))\|_{H^{s}\times H^{s+\frac12}}=+\infty} \, .
\end{equation}
Then, there exists $T_0=T_0(\|(\eta_0,u_0)\|_{H^s \times H^{s+\frac12}})$ such that $T^{\star} > T_0$ and 
\begin{equation} \label{apriori.1}
 \sup_{t \in [0,T_0]} \|(\eta,u)(t)\|_{H^s\times H^{s+\frac12}} \le c\|(\eta_0,u_0)\|_{H^s \times H^{s+\frac12}} \, ,
\end{equation}
for some positive constant $c$.
\end{lemma}

\begin{proof} 
Let us define \[ \widetilde{T}:= \sup\Big\{ T \in (0,T^{\star}) \ : \ \sup_{t \in [0,T]} \| (\eta,u)(t)\|_{H^s \times H^{s+1/2}}^2 \le 8 \| (\eta_0,u_0)\|_{H^s \times H^{s+1/2}}^2 \Big\} \, .\]
Note that $\widetilde{T}<T^{\star}$, otherwise it would contradict the blow-up alternative \eqref{blow-up_alt}.

We  define $T_0:= \min\{T_1,T_2\}$ where \[ T_1=\frac1{C_1} \log \left( 1+\frac1{1+C_1\|(\eta_0,u_0)\|^2_{H^s\times H^{s+1/2}}}\right), \quad  T_2= \frac{h_0}{C_2 \big(1+\|\eta_0\|_{H^s} \big) \|u_0\|_{H^{s+1/2}}} \]
 and $C_1, \, C_2$ are two large positive constants to be fixed in the proof. 
 
 Assume by contradiction that $\widetilde{T}<T_0$, otherwise we are done. By continuity, we have that 
 \begin{equation} \label{apriori.10}
 \sup_{t \in [0,\widetilde{T}]}\| (\eta,u)(t)\|_{H^s \times H^{s+1/2}}^2 \le 8 \| (\eta_0,u_0)\|_{H^s \times H^{s+1/2}}^2\, .
 \end{equation}
We first verify that  the non-cavitation condition \eqref{noncavitation.1} holds on $[0,\widetilde{T}]$. On the one hand, since $\eta_0$ satisfies \eqref{noncavitation}, it follows from the fundamental theorem of calculus that 
\begin{equation*}
\eta(x,t)+1 =\eta_0(x)+1+\int_0^t \partial_t \eta(x,s) ds \ge h_0-\widetilde{T} \sup_{s\in [0,\widetilde{T}]} \| \partial_t \eta(s) \|_{L^{\infty}_x} \, ,
\end{equation*}
for all $t \in [0,\widetilde{T}]$.
On the other hand, we estimate trivially by using the first equation in \eqref{WBbis}, the Sobolev embedding and, then \eqref{apriori.10}, that 
\begin{equation*}
 \| \partial_t \eta \|_{L^{\infty}_x} \le c \big(1+\|\eta\|_{H^s} \big) \|u\|_{H^s} \le c  \big(1+\|\eta_0\|_{H^s} \big) \|u_0\|_{H^{s+1/2}}\, ,
 \end{equation*}
$\forall t \in [0,T_1]$. Thus, by recalling that $\widetilde{T}<T_2$ and by choosing $C_2$ large enough, we deduce from the above analysis that $\eta(x,t)+1 \ge h_0/2$ on $[0,\widetilde{T}]$. This and a similar argument together with the Sobolev embedding $\eta_0 \in H^s(\mathbb R) \hookrightarrow L^{\infty}(\mathbb R)$ show that the condition \eqref{noncavitation.1} hold on the time interval $[0,\widetilde{T}]$.

Let $y(t):= E(\eta,u)(t)$ denote the modified energy defined in \eqref{EE.1}. Then
\eqref{EE.3} leads to the inequality $y'(t) \le c\big( y(t)+y^2(t) \big)$, which can be integrated
to obtain
\begin{displaymath} 
y(t) \big(1-\frac{y_0}{1+y_0}e^{ct}\big) \le \frac{y_0}{1+y_0} e^{ct}, \quad \text{if} \quad \frac{y_0}{1+y_0}e^{ct}<1 \, ,
\end{displaymath}
on $[0,\widetilde{T}]$. Since $\widetilde{T} <T_1$, we get by choosing $C_1=C_1(h_0)$ large enough and by using  \eqref{EE.2}  that
\begin{equation*} 
\frac{y_0}{1+y_0}e^{ct} \le \frac{1+2y_0}{2(1+y_0)}<1 \quad \text{and} \quad \| (\eta,u)(t)\|^2_{H^s \times H^{s+1/2}} \le 4 \| (\eta_0,u_0)\|^2_{H^s \times H^{s+1/2}} ,
\end{equation*}
for all $t \in [0,\widetilde{T}]$. Thus, we deduce by continuity, that there exists some $\widetilde{T}<T<T^{\star}$ such that 
$\| (\eta,u)(T)\|^2_{H^s \times H^{s+1/2}} \le 6 \| (\eta_0,u_0)\|^2_{H^s \times H^{s+1/2}}$. This contradicts the definition of $\widetilde{T}$. Therefore, $\widetilde{T}<T_0$, which concludes the proof of Lemma \ref{apriori}.
\end{proof}

\medskip
With the \textit{a priori} estimate in hand, the complete proof of the existence would then result from a standard compactness argument implemented on a regularized version
of the system. The uniqueness is a consequence of the estimates for the difference of two solutions \eqref{Est_diff.3}. The strong continuity in time and the continuity of the flow would result from an application of the Bona-Smith argument \cite{BoSm} (we refer to \cite{KePi} for a detailed demonstration of the use of the Bona-Smith argument in the context of the modified energy).

\section{The two-dimensional case}
In this section, we comment briefly on the changes to adapt the proof in the two-dimensional setting. By denoting $\textbf{u}=(u_1,u_2)^{T}$, we reformulate the system \eqref{WB2d} as 
\begin{equation} \label{WB2dbis}
\left\{\begin{array}{l}
\partial_t \eta-(\widetilde{\mathcal{M}}(D)+1)(1-\Delta)(\mathcal{R}_1u_1+\mathcal{R}_2u_2)+\partial_{x_1}(\eta u_1)+\partial_{x_2}(\eta u_2)=0 \, , \\ 
\partial_t u_1+\frac12\partial_{x_1}(u_1^2+u_2^2)=0 \, , \\
\partial_t u_2+\frac12\partial_{x_2}(u_1^2+u_2^2)=0 \, ,
\end{array}\right.
\end{equation}
where $\mathcal{R}_j$ denote the Riesz transform and $\widetilde{\mathcal{M}}(D)$ is the Fourier multiplier associated to the symbol 
$\widetilde{M}(\xi)=\tanh |\xi|-1$, $\xi=(\xi_1,\xi_2)$ and $|\xi|=\sqrt{\xi_1^2+\xi_2^2}$. Note that the pointwise estimate $\big|\tanh|\xi|-1 \big| \le e^{-|\xi|}, \quad \forall \, \xi \in \mathbb R^2$, holds true.

We derive an energy estimate (analogous to Proposition \ref{EE}) at the level $(\eta,u_1,u_2) \in H^{s}(\mathbb R^2) \times H^{s+\frac12}(\mathbb R^2) \times H^{s+\frac12}(\mathbb R^2)$, $s>\frac52$ by using the modified energy 
\begin{displaymath} 
\begin{split}
E^{s}(\eta,u_1,u_2)(t)&=\frac12 \|\eta(t)\|_{H^s}^2+\frac12 \|u_1(t)\|_{H^{s+\frac12}}^2+\frac12 \|u_2(t)\|_{H^{s+\frac12}}^2 \\ & \quad+\frac12\int \eta \big((J^s_xu_1)^2 + (J^s_xu_2)^2\big)(t)\, .
\end{split}
\end{displaymath}
Note that instead of using the identity $D^1_x=\mathcal{H}\partial_x$, we use in a crucial way the identities $\mathcal{R}_j\Delta=D^1_x\partial_{x_j}$, $j=1,2$, to cancel out the linear terms. The modified energy is then used to handle the nonlinear term in the first equation of \eqref{WB2dbis}. To deal with the nonlinear terms in the second and third equations of \eqref{WB2dbis}, we use  Lemma 4.2. in \cite{LiPiSa}\footnote{As already observed in the introduction, it is for this reason that we need to make the curl-free assumption on the velocity $\textbf{u}$.}. 

The proof of the uniqueness is very similar to Proposition \ref{Est_diff}.

\vskip 0.05in
\noindent
{\bf Acknowledgments.} 
This research was supported by the Bergen Research Foundation Foundation (BFS),
the Research Council of Norway,  and the University of Bergen. 
The authors would also like to thank Jean-Claude Saut, Vincent Duch\^ene and Mats Ehrnstr\"om for helpful comments on a preliminary version of this work.
They also thank the anonymous referee for his comments and suggestions.

\bibliographystyle{amsplain}

\end{document}